\documentclass[11pt]{amsart}

\usepackage{amsgen,amsthm,amstext,amsbsy,amsopn,amsfonts,amssymb, enumerate}

\usepackage{amsmath}

\newcommand\la{\langle}
\newcommand\ra{\rangle}

\newcommand\ff{{\mathfrak f}}
\newcommand\ggo{{\mathfrak g}}
\newcommand\hh{{\mathfrak h}}

\newcommand\kk{{\mathfrak k}}

\newcommand\nn{{\mathfrak n}}

\newcommand\sso{{\mathfrak{so}}}
\newcommand\vv{{\mathfrak v}}

\newcommand\zz{{\mathfrak z}}

\newcommand\CC{\mathbb C}

\newcommand\RR{\mathbb R}

\newcommand\so{{\mathrm{so}}}

\newcommand\ad{\operatorname{ad}}

\newcommand\lf{\operatorname{F}}

\theoremstyle{plain}

\newtheorem{thm}{Theorem}[section]
\newtheorem{lem}[thm]{Lemma}
\newtheorem{prop}[thm]{Proposition}
\newtheorem{cor}[thm]{Corollary}

\theoremstyle{definition}
\newtheorem{defn}[thm]{Definition}
\newtheorem{rem}[thm]{Remark}
\newtheorem{example}[thm]{Example}

\begin{document}

\title[Symplectic structures on Low dimensional   2-step nilmanifolds]
{Symplectic structures on Low dimensional   2-step nilmanifolds}

\author{Gabriela P. Ovando, Mauro Subils}

\thanks{{\it (2000) Mathematics Subject Classification}: 53D05, 57T15, 22E25. }

\thanks{{\it Key words and phrases}: Symplectic structures,  2-step nilmanifolds, closed 2-forms.
 }

\thanks{Partially supported by  SCyT (UNR)}

\address{ Departamento de Matem\'atica, ECEN - FCEIA, Universidad Nacional de Rosario.   Pellegrini 250, 2000 Rosario, Santa Fe, Argentina.}

\

\email{gabriela@fceia.unr.edu.ar}

\email{subils@fceia.unr.edu.ar}

%\date{\today}

\begin{abstract}  The aim of this work is the study of symplectic structures  on 2-step nilmanifolds. We concentrate in the closedness condition, proving that the existence of a  closed 2-form of type II is necessary to get a  symplectic structure. In low dimensions, this condition is sufficient in most cases. %Moreover, it is proved that up to dimension six this condition is sufficient. 
\end{abstract}

\maketitle

 \noindent\section{Introduction}

 A 2-step nilmanifold is a smooth  manifold $M= \Lambda \backslash N$,  where $N$ is a simply connected 2-step nilpotent Lie group and $\Lambda$  is co-compact discrete subgroup of $N$. After a result of Nomizu \cite{N}    any symplectic structure on the  nilmanifold $M$ is cohomologous to an invariant one on $N$. In this work we search for symplectic structures, by concentrating the attention to the closedness condition of the corresponding 2-form. Closed 2-forms on $M=\Gamma \backslash N$ are called {\em magnetic fields}, and they are  deeply related to  magnetic trajectories (see for instance \cite{OS2} and references herein). In fact, magnetic trajectories are curves on a Riemannian manifold $(M, \la \,, \, \ra)$,  solutions of the equation \begin{equation}\label{magneteq}
 	\nabla_{\gamma'} {\gamma'} = \lf{\gamma'},
 \end{equation}
 where $\nabla$ is the Levi-Civita connection on $M$ and $\lf$ is a (1,1)- tensor giving rise to a closed 2-form $\omega=\la \lf\cdot, \cdot \ra$, which could be degenerate. This is the focus of our research.

 Symplectic 2-forms on  homogeneous spaces were studied by different authors (see for instance \cite{DM,Ch,FG,GB, Ov}, and  specifically on nilpotent Lie groups, see results for instance in \cite{BG,dB1,DT,GR}. Also a construction by a double extension procedure was given in \cite{DM}.

Since the problem of determining
whether an arbitrary Lie algebra admits symplectic structures is difficult in general, it was 
attacked by using different strategies. In a case-by-case study, after the classification  of nilpotent Lie algebras, one gets results in dimension six. See \cite{GR} for some results in dimension eight.  Several subfamilies of nilpotent
symplectic Lie algebras have been described. For example, symplectic filiform algebras \cite{Bu,GMK} or free nilpotent Lie algebras in  \cite{dB2}. 

 Pouseele and Tirao showed that for a nilpotent Lie algebra associated to a graph, the condition 
 $$2 \dim [\nn, \nn]\leq \dim \nn$$
  implies the existence of symplectic structures \cite{PT}. Moreover, the graph encodes geometrical  information. 
 Previously,  	Dotti and Tirao in \cite{DT} proved that if the 2-step nilpotent Lie group $N$ admits a  symplectic structure, then 
  $$2\dim[\nn, \nn] \leq \dim \nn +  1.$$

 As said, the focus here is concentrated on closed 2-forms.  By studying their existence, it was  proved in \cite{OS} that 
   if $\nn$ is non-singular and
    $$\dim \nn > 3 \dim [\nn, \nn]$$
    then any closed 2-form $\omega$ on $\nn$ satisfies
    $$ (*)\quad \omega(Z,U)=0, \qquad \mbox{ for all } Z\in \zz, U\in \nn.$$
    Notice that in the non-singular case, the commutator $[\nn, \nn]$ coincides with the center of the Lie algebra $\zz$.  The condition above is equivalent to saying that for any metric, the corresponding skew-symmetric map induced by $\omega$ preserves the decomposition (**), see below. 
     
     Condition (*) above is the key of our results. In term of forms, it says that there is no 2-form $\omega$ on $\nn$ such that $\omega(\zz, \zz)=0$ but $\omega (\zz, \nn)\neq 0$. This is what we call a 2-form of type II. Now, asking for  symplectic structures, we get: 
     
     \smallskip
     
     Lemma. {\em  If a 2-step nilpotent Lie group admits a symplectic structure, then it admits a closed 2-form of type II. }
     
     \smallskip
     
  This observation is the start point to study the family of closed 2-forms of type II, which is done with the help of a metric. Recall that on a 2-step nilpotent Lie group $N$ equipped with a left-invariant metric, one makes use of a natural decomposition of the corresponding Lie algebra
 	$$(**) \qquad \nn=\vv \oplus \zz, \quad \mbox{where } \vv=\zz^{\perp}. $$
% which is obtained	for a fix a metric $\la\,,\,\ra$ on the Lie algebra $\nn$. 
 
 	Now, given a left-invariant 2-form $\omega$ on $N$, one can find a skew-symmetric map $\lf:\nn\to \nn$, such that $\omega =
 \la \lf \cdot, \cdot \ra$. We  study the existence of symplectic structures, in terms of conditions on the map $\lf$. Any such linear map $\lf$ decomposes as $\lf=\lf_1 +\lf_2$, where $\lf_1$ preserves the decomposition above, while $\lf_2$ makes $\lf(\zz)\subseteq \vv$ and $\lf(\vv)\subseteq \zz$. This corresponds to 2-forms of type I or II respectively. 
 
 The closedness condition of $\omega$ gives two closedness conditions for the 2-forms of type I and II.  Explicitly, given the constant structures of the Lie algebra $C_{ij}^k$ where $v_1, \hdots, v_n$ is an orthonormal basis of $\vv$, the closedness condition for type II reduces
 to 
 $$ \sum_{s=1}^m (C_{ij}^s b_{sk} +  C_{jk}^s b_{si} + C_{ki}^s b_{sj})=0, $$
 where $(b_{sj})$ gives a $m\times n$ matrix and $m=\dim \zz$. One searches for  non-trivial solutions. 
 	
 	Looking for the symplectic structures on low-dimensional Lie algebras, we prove
 	
 	 \smallskip
 	
  {\em Let $N$ denote a nilpotent Lie group of dimension $2n$, $n\leq 3$. With exception of the trivial extension of the Heisenberg Lie algebra of dimension five, any such Lie group admits a symplectic structure if and only if it admits a closed 2-form of type II.} 
 	
 	\smallskip
 	
 Finally we show examples in higher dimensions, where the result above is no longer true. This is provided by free 2-step nilpotent Lie algebras in $n$ generators with $n\geq 4$. In fact, these Lie algebras admit non-trivial  closed 2-forms of type II but no symplectic structures.

\section{Lie groups of step two with a left-invariant metric}\label{general}

In this section we recall basic facts on 2-step nilpotent Lie groups equipped with a left-invariant metric. Firstly recall that a  Lie group is called 2-step nilpotent if its Lie algebra is 2-step nilpotent, that is, the Lie bracket satisfies $[[U,V], W]=0$ for all $U,V,W\in \nn$. Throughout this paper Lie groups, so as their Lie algebras are considered over $\RR$.

\begin{example} \label{exa1} The smallest dimensional non-abelian 2-step nilpotent Lie group is the Heisenberg Lie group of dimension three, $H_3$. Its Lie algebra $\hh_3$  is spanned by vectors $e_1, e_2, e_3$ satisfying the non-trivial Lie bracket relation
	$$[e_1,e_2]=e_3.$$
	The Lie group $H_3$ can be modeled on $\mathbb R^3$ equipped with the product operation given by
	$$(v_1,z_1)(v_2,z_2)=(v_1+v_2, z_1+z_2+\frac{1}2 v_1^tJv_2),$$
	where $v_i=(x_i,y_i)$, i=1,2 and $J:\RR^2 \to \RR^2$ is the linear map $J(x,y)=(y, -x)$. % By using this, usual computations show that a basis of left-invariant vector fields is given at $p=(x,y,z)$ by
%	$$e_1(p)=\partial_x -\frac12 y \partial_z, \quad e_2(p)=\partial_y+\frac12 x \partial_z, \quad e_3(p)=\partial_z.$$
Another presentation of the Heisenberg Lie group is given by $3\times3$-triangular real  matrices with 1's on the diagonal with the usual multiplication of matrices.  
	\end{example} 

A Riemannian  metric $\la\,,\,\ra$ on the Lie group  $N$ is called {\it left-invariant} if  translations on the left by elements of the group are isometries. Thus, a left-invariant metric  is determined at the corresponding Lie algebra $\nn$, usually identified with the tangent space at the identity element $T_eN$. We denote also by  $\la\,,\,\ra$ the metric on $\nn$. 

Whenever the Lie algebra $\nn$ is 2-step nilpotent, its commutator $C(\nn)$  is contained in the center $\zz$. So the metric $\la\,,\,\ra$ determines an orthogonal decomposition as vector spaces on the Lie algebra: 
\begin{equation}\label{decomp2}
	\nn=\vv \oplus \zz, \quad \mbox{ where }\quad \vv =\zz^{\perp}.
\end{equation}
% The subspaces $\vv$ and $\zz$ induce left-invariant distributions on $N$, denoted by $\mathcal V$ and $\mathcal Z$, just by making use of the tranlations on the left.  Recall that a distribution on a smooth manifold $M$ is a choice a subspace $D_p\subseteq T_pM$ for every $p\in M$. 

The decomposition in Equation \eqref{decomp2} induces the skew-symmetric maps $j(Z):\vv \to \vv$, for every $Z\in \zz$,  implicitly defined by 
\begin{equation}\label{j}
	\la Z, [V,W]\ra =\la j(Z) V, W \ra \qquad \mbox{ for all } Z\in \zz, V, W\in \vv. 
\end{equation}

Note that $j:\zz \to \mathfrak{so}(\vv)$ is a linear map, which has a corresponding kernel. As said  the commutator $C(\nn)$ is contained in the center and  one has the splitting 
$$\zz=C(\nn)\oplus \ker(j)$$ as orthogonal direct sum of vector spaces. In fact, 
\begin{itemize}
	\item since $\la Z, [U,V]\ra=0$ for all $U,V\in \vv$ and $Z\in \ker(j)$,  then $\ker(j)\perp C(\nn)$.
	\item $\dim \zz = \dim \ker(j)+ \dim C(\nn)$. 
	\item The restriction $j:C(\nn) \quad \mapsto \quad \mathfrak{so}(\vv)\quad \mbox{is injective}.$
\end{itemize} 
%In fact, assume $j(Z)=j(\bar{Z})$ for $Z, \bar{Z}\in C(\nn)$. Then $j(Z-\bar{Z})=0$, so that $Z-\bar{Z}\in \ker(j)\cap C(\nn)$. Thus $Z-\bar{Z}=0$. 

See the proof of the next result in Proposition 2.7 in \cite{Eb}. 
 
 \begin{prop} \cite{Eb} Let $(N, \la\,,\,\ra)$ denote a 2-step nilpotent Lie group with a left-invariant metric. Then 
 	\begin{itemize}
 		\item the subspaces $\ker j$ and $C(\nn)$ are commuting ideals in $\nn$. 
 		\item Let $E=\exp(\ker(j))$. Then $E$ is the Euclidean de Rham factor of $N$ and $N$ is isometric to the Riemannian product of the totally geodesic submanifolds $E$ and $\bar{N}$ where $\bar{N}=\exp(\vv \oplus C(\nn))$. 
 	\end{itemize}
 	
 	\end{prop}

 \begin{example}
 	Let $\hh_3$ denote the Heisenberg Lie algebra of dimension three with basis $e_1, e_2, e_3$ as in Example \ref{exa1}. The canonical  metric makes   this basis  orthonormal. It is not hard to see that the center is the subspace
 	 $\zz=span\{e_3\}$, while its orthogonal complement is  the subspace $\vv=span\{e_1, e_2\}$ and moreover the map $j:\zz\to \so(\vv)$ is generated by 
 	$$j(e_3)=\left( \begin{matrix}
 		0 & -1\\
 		1 & 0 
 	\end{matrix}\right), 
 $$
 in the basis $e_1, e_2$ of $\vv$. 
 \end{example}

Let $\Gamma\subset N$ denote a discrete subgroup of the nilpotent Lie group $N$, such that the quotient $M=\Gamma \backslash N$ is compact. This is called a {\em compact nilmanifold}. Important geometrical properties of $M$ are obtained from the Lie algebra of $N$, $\nn$. In fact, by a 
 theorem of Nomizu \cite{N} every de Rahm cohomology group $H^i(M, \RR)$ is isomorphic to the group $H^i(\nn)$. 

As a consequence, results obtained at the Lie algebra level and related to n-forms, are applicable for the corresponding Lie group and their quotients.

\begin{defn}
A  2 -step nilpotent real Lie algebra $\nn$ with center $\zz$ is called {\em non-singular}  if  $\ad(X): \mathfrak{n} \rightarrow \mathfrak{z}$ is onto for any $X \notin \mathfrak{z}$ \cite{Eb}. The corresponding 2-step nilpotent Lie group will be called non-singular. 

\end{defn}

See the next examples of non-singular Lie algebras. 

\begin{example} {\it Heisenberg Lie algebras.} \label{ExHeis} Let $n\geq 1$ be any integer and let $X_1,Y_1, X_2, Y_2,$ $ \hdots, X_n, Y_n$ be any basis of a real vector space $\vv$ isomorphic to $\RR^{2n}$. Let $Z$ be an element generating a one dimensional space $\zz$. Define a Lie bracket by $[X_i, Y_i]=-[Y_i, X_i]=Z$ and the other Lie brackets by zero. The Lie algebra $\hh_{2n+1}=\vv \oplus \zz$ is the $(2n+1)$-dimensional Heisenberg Lie algebra.   	
\end{example}

By making use of the  maps $j(Z):\vv \to \vv$ defined in Equation \eqref{j}, one may described the non-singularity notion. See \cite{Eb}. 

\begin{prop}
Let $\nn$ denote a 2-step nilpotent Lie algebra. The following properties are equivalent:
\begin{enumerate}
	\item $\nn$ is non singular, 
	\item For every inner product $\la\,,\,\ra$ and every non-zero element $Z\in \zz$ the map $j(Z)$ is non-singular. 
	\item For some inner product $\la\,,\,\ra$ and every non-zero element $Z\in \zz$ the map $j(Z)$ is non-singular. 
\end{enumerate}
\end{prop}

Non-singular  Lie algebras are  known as {\em fat} algebras because they are the symbol algebra [Definition 3.1.7 in \cite{Ca}] or nilpotentization \cite{M} of fat distributions.

A 2-step nilpotent Lie algebra $\nn = \vv\oplus \zz$ equipped with a inner product $\la\,,\,\ra$ is {\em singular} if $j(Z)$ is singular for every nonzero $Z \in \zz$.

The Lie algebra $\nn$ is {\em almost non-singular} if $j(Z)$ is non-singular for every $Z$ in an open dense subset of $\zz$. 

Every 2-step nilpotent Lie algebra is non-singular, almost non-singular
or singular. See \cite{Eb}. 

Given a metric on a 2-step nilpotent Lie algebra $\nn$, if two nonzero elements $Z, Z'\in \zz$  can be found such that
$j(Z)$ is non-singular and $j(Z')$ is singular, then the Lie algebra $\nn$ is almost non-singular.

Among other examples, see \cite{LO,LT}, a family of non-singular Lie algebras is provided by H-type Lie algebras, which are defined as follows. 

Let $(\nn, \la\,,\,\ra)$ denote a 2-step nilpotent Lie algebra equipped with a metric. If the map $j(Z):\vv\to\vv$ is orthogonal for every $Z\in\zz$ with $\la Z, Z \ra =1$,  then the Lie algebra $\nn$ is a {\em Lie algebra  of type H  } \cite{K} (known as $H$-type Lie algebras). Equivalently, the 2-step nilpotent Lie algebra $\nn$ is of {\em type H } if and only if   $$j(Z)^2=-\la Z, Z\ra Id, \qquad \mbox{ for every }Z\in\zz, $$
which is equivalent to 
$j(Z) j(\widetilde{Z})+ j(\widetilde{Z})j(Z) =-2\la Z, \widetilde{Z}\ra Id$, for $Z, \widetilde{Z}\in \zz$. 
By making use of this identity one proves that 
$$[X,j(Z)X]=\la X, X\ra Z$$
for every $X\in\vv$ and $Z\in\zz$.

In particular, the real, complex and quaternionic Heisenberg algebras are examples of Lie algebras of type H.

\section{Left-invariant 2-forms  and magnetic fields}\label{closedforms}

In this section we study closed left-invariant 2-forms, called {\em magnetic fields},  on any 2-step nilpotent Lie group. We shall see that the closedness condition imposes several restrictions. 

  Let $\omega$ denote a left-invariant 2-form on a  Lie group $(N, \la\,,\,\ra)$ equipped with a left-invariant metric. It is {\em closed}  if 
  $$\omega([U,V],W) + \omega([V,W],U)+ \omega([W,U],V)=0\qquad \mbox{ for all }U,V,W\in \nn.$$
  
  Note that a left-invariant 2-form on $N$ is determined by its values at the identity. Thus, we say this is a 2-form on $\nn$. 
  The closed 2-form $\omega$ is {\em symplectic} whenever it is non-degenerate, that is, if $\omega(U,V)=0$ for all $V\in \nn$, then $U=0$. 
  
 Let $\la\,,\,\ra$ denote the metric on  $\nn$. Let $\omega$ be a left-invariant 2-form on $N$, then   there exists a unique skew-symmetric endomorphism $\lf:\nn \to \nn$ satisfying
$$\omega(X,Y)= \la \lf(X), Y\ra, \quad \mbox{ for all } X,Y\in \nn.$$

Conversely, let $\lf:\nn \to \nn$ denote a skew-symmetric endomorphism on the Lie algebra.  Define the associated 2-form $\omega$ as   $\omega(X,Y)= \la \lf(X),Y\ra$. %Thus, very left-invariant 2-form on $(N, \la\,,\,\ra$) is in correspondence with a skew-symmetric endomorphism $\lf\in\sso(\nn)$. 
When $\omega$ is closed, such skew-symmetric map $\lf$ is known as a {\em Lorentz force}.

Now,  the condition of asking the 2-form $\omega$ to be closed  is equivalent to ask the skew-symmetric map $\lf$ to satisfy the equation:
\begin{equation}\label{clos}	\la \lf(U), [V,W]\ra + \la \lf(V), [W,U]\ra + \la \lf(W), [U,V]\ra=0, \qquad \mbox{ for all } U,V,W\in \nn.
	\end{equation}
Furthermore, $\omega$ non-degenerate, if and only if $F$ is non-singular. 

Assume now  that the 2-step nilpotent Lie group $N$ has  Lie algebra $\nn$, which decomposes into the orthogonal splitting $\nn=\vv \oplus \zz$ as in Equation  \eqref{decomp2}. 
Write $\lf_{\vv}=\pi_{\vv}\circ \lf$ and $\lf_{\zz}=\pi_{\zz}\circ \lf$, where $\pi_{\vv}:\nn\to \vv$ and $\pi_{\zz}:\nn\to \zz$ denote  the corresponding orthogonal projections onto the subspaces $\vv$ and $\zz$, respectively.  So, the 2-form $\omega$ associated to $\lf$ is closed if and only if the following conditions hold
%\begin{equation}\label{closed2}
$$({\rm C1})\qquad \qquad \qquad 	\lf_{\zz}(Z)\in C(\nn)^{\perp}, \qquad \mbox{ for all } Z\in \zz,  
$$
and

$$({\rm C2}) \qquad 	\la \lf_{\zz}(U), [V,W]\ra + \la \lf_{\zz}(V), [W,U]\ra + \la \lf_{\zz}(W), [U,V]\ra=0, \quad \mbox{ for all } U,V,W\in \vv,
$$
which can be obtained by analyzing Condition \eqref{clos} in terms of the projections onto the subspaces $\vv$ and $\zz$.

\begin{rem}\label{ztrivial}
	For any exact form, the corresponding Lorentz map satisfies $\lf_{\zz} \equiv 0$. Furthermore,   any skew-symmetric map on the Lie algebra $\lf:\nn\to \nn$   such that the projection $\lf_{\zz}\equiv 0$ trivially satisfies Equations (C1) and (C2). 
	% since $\la\lf_{\zz}(U), [V,W]\ra\equiv 0$. Examples of such Lorentz forces $\lf$ are provided for those corresponding to  exact 2-forms.
\end{rem}

The skew-symmetric map $\lf:\nn \to \nn$  decomposes as
$$\lf=\lf_1+\lf_2,$$
where $\lf_1$ and $\lf_2$ are skew-symmetric maps such that, with respect to the orthogonal splitting  $\nn=\vv\oplus\zz$ in Equation \eqref{decomp2}, one has:
\begin{itemize}
	\item $\lf_1$ preserves the decomposition: $\lf_1(V+Z)=\lf_\vv(V) + \lf_\zz(Z)$, 
	\item $\lf_2$ interchanges the subspaces $\vv$ and $\zz$: $\lf_2(V+Z)=\lf_\zz(V) + \lf_\vv(Z)$ for all $V\in\vv$, $Z\in\zz$.
\end{itemize} 

In fact,  for any skew-symmetric map  $\lf\in \sso(\nn)$, one has the decomposition $\lf=\lf_{\zz}+\lf_{\vv}$ an so,  
$$\lf(V+Z)= (\lf_{\vv}(V)+\lf_{\zz}(Z))+(\lf_{\vv}(Z)+\lf_{\zz}(V)), \quad \mbox{ for all } V\in \vv, Z\in \zz. $$
%Take the maps $\lf_1$ and $\lf_2$ respectively, given by

 % $\lf_1(V+Z)=\lf_{\vv}(V)+\lf_{\zz}(Z)$ and 
%$\lf_2(V+Z)=\lf_{\vv}(Z)+\lf_{\zz}(V)$, for all $V+Z\in \nn$. 

The skew-symmetry property from $F$ implies that  both $\lf_1$ and $\lf_2$ are skew-symmetric. 

Notice that
\begin{itemize}
	\item $\lf_1$ trivially satisfies Condition (C2) and 
	\item $\lf_2$ trivially satisfies Condition (C1).
\end{itemize}

Thus the skew-symmetric map $\lf$ gives rise to a closed 2-form if and only if for both $\lf_1$ and $\lf_2$ it holds:
\begin{itemize}
	\item $\lf_1$  satisfies Condition (C1), that is $\lf_1(\zz)\subseteq \ker(j)$ and 
	\item $\lf_2$  satisfies Condition  (C2).
\end{itemize}

\begin{defn}\label{def1}
	Let $\lf$ denote a skew-symmetric map on a 2-step nilpotent Lie algebra $(\nn, \la\,,\,\ra)$. We say that
	\begin{enumerate}
		\item $\lf$ is of type I, if $\lf$ preserves the decomposition $\vv \oplus \zz$ (so, $F=F_1$)
\item $\lf$ is of type II, if $\lf(\vv)\subseteq \zz$ and $\lf(\zz)\subseteq \vv$ (so, $\lf=\lf_2$ above).
	\end{enumerate}
\end{defn}

Next, we shall study exact 2-forms. Start with a left-invariant 1-form $\eta$. Consider the linear isomorphism between the Lie algebra $\nn$ and its dual space $\nn^*$ given by the metric, that is, sending $U \to \ell_U$, where $\ell_{U}(V)=\la U, V\ra$. 

By considering the decomposition of the Lie algebra  $\nn$ given in  Equation \eqref{decomp2} one can write any left-invariant $1$-form $\eta$ as $\eta=\ell_{Z+V}$. Easily one verifies that the differential follows
$$d\ell_{Z+V}(U,V)=\la Z,[U,W]\ra=\la j(Z)U,W\ra.$$
The kernel of the differential operator $d:\nn^* \to \Lambda^2\nn^*$ contains  the subspace $\{\ell_{{V}}, \,\mbox{ with } V\in \vv\}$. For the description of  the rank of $d$ notice that  $d\eta\neq 0$ if and only if $j(Z)\neq 0$. This implies $Z$ belongs to $C(\nn)$. Thus,   for the 2-form $d\ell_{Z}$, the corresponding Lorentz force $F$ is given by $\lf=j(Z)$ for some non-trivial $Z\in C(\nn)$. 

We already proved the next result. 

\begin{prop} \label{prop1} Let $(N, \la\,,\,\ra)$ denote a 2-step-nilpotent Lie group equipped with a left-invariant metric and Lie algebra $\nn$ with orthogonal splitting $\nn=\vv\oplus \zz$ as in \eqref{decomp2}. Let $\lf:\nn \to \nn$ denote a linear map. Write the map  $\lf$  as 
	$$\lf=\lf_1+\lf_2, $$
	where $\lf_1(\vv)\subseteq \vv$ and $\lf_1(\zz)\subseteq \zz$,  while $\lf_2(\zz)\subseteq \vv$ and $\lf_2(\vv)\subseteq \zz$. Then
	\begin{enumerate}[(i)]	
		\item The skew-symmetric map $\lf$ gives rise to a closed 2-form if and only if both maps $\lf_1$ and $\lf_2$ are skew-symmetric and
			\begin{itemize}	\item $\lf_1$  satisfies \rm{Condition } \rm{(C1)} and 
		\item $\lf_2$  satisfies \rm{Condition} \rm{(C2)}.

		\end{itemize}
		\item The 2-form associated with $\lf$ is exact
		if and only if there is $Z\in C(\nn)$ such that $\lf = j(Z)$.
		% Exact left-invariant 2-forms are in one-to-one correspondence with skew-symmetric maps $j(\widetilde{Z})$ with $\widetilde{Z}\in C(\nn)$. 
	\end{enumerate}
\end{prop}
\begin{rem}\label{remtypeIandII} Note that if the 2-step nilpotent Lie group $N$ has no Euclidean factor, then $\zz=C(\nn)$. Thus,   any skew-symmetric map $\lf$ of type I satisfies $\lf_{\zz}\equiv 0$.
	%, so that it always gives rise to a closed 2-form. Moreover, the linear map $\lf$ is trivial on the center.  
	
	On the other hand, if $N$ has Euclidean factor then $ker(j)\neq\{0\}$ and any non-trivial skew-symmetric map $\lf:\nn\to\nn$ that satisfies $\lf(\vv)\subseteq ker(j)$ and $\lf(\zz)\subseteq \vv$ give rise to a closed left-invariant 2-form of type II on $N$.  
\end{rem}

\begin{example} Let $V_0+Z_0$ be any element on a 2-step nilpotent Lie algebra $\nn$. Then a natural choice for $\lf$ is the skew-symmetric part of $\ad(V_0+Z_0)$: $\ad(V_0+Z_0)-\ad(V_0+Z_0)^*=\ad(V_0)-\ad(V_0)^*$, where $\ad(X)^*$ denotes the adjoint of  $\ad(X)$ with respect to the metric. Notice that $\ad(V_0)(\vv)\subseteq \zz$ and $\ad(V_0)^*(\zz)\subset \vv$. Thus $\ad(V_0)-\ad(V_0)^*$ gives rise to a closed 2-form if and only if Condition (C2) holds, equivalently:
	$$\la [V_0, V], [U,W]\ra + \la [V_0,U], [W,V]\ra + \la [V_0,W], [V, U]\ra=0\mbox{ for all } U,V,W\in \vv.$$
Which kind of elements $V_0\in \vv$ may satisfy this equation?
Take the Heisenberg Lie algebra of dimension 2n+1, $n\geq 2$, in Example \ref{ExHeis}. Assume $V_0=\sum_i x_iX_i + \sum_i y_iY_i$. By taking $U=Y_i, V=X_j$ and $W=Y_j$ with $i\neq j$, one gets $x_i=0$ for every $i$. Analogously, by taking $U=X_i$, one finally obtains $V_0=0$.  %This situation is an example of a more general proved in  Lemma \ref{closedF1}.
	\end{example}
	Notice that if a Lie group $N$ admits a symplectic structure, then the dimension of $N$ is an even integer.

	\begin{lem} Let $N$ denote a 2-step nilpotent Lie group admitting a  symplectic structure. Then $N$ admits a left-invariant Lorentz force of type II.
	\end{lem}
	
	\begin{proof} Choose a left-invariant metric on $N$ and fix this metric at the Lie algebra level. Take the decomposition $\nn=\vv \oplus\zz$ and in Equation \eqref{decomp2}. Take the linear map $\lf\in \sso(\nn)$ such that  $\omega(u,v)=\la \lf(u), v\ra$, and write $\lf=\lf_1+\lf_2$, where $\lf_1$ preserves the decomposition, but $\lf_2(\zz)\subseteq \vv$ and $\lf_2(\vv)\subseteq \zz$. 
	Assume $\lf_2$ is trivial, so $\lf=\lf_1$ is of type $I$. The closedness condition says that $\lf(\zz)\subseteq \ker(j)\subsetneq \zz$. But since $\lf(\vv)\subseteq \vv$, it follows that $\lf$ cannot be non-singular, which is a contradiction. 
		\end{proof}
		
	\begin{rem} The proof of the previous lemma shows that a Lorentz force of type I never induces a symplectic form. 
		
	Moreover, let $\nn$ be a nilpotent Lie algebra and let $\eta$ denote a $1$-form. Thus, $d\eta(Z,U)=0$ for every $Z\in\zz$, $U\in\nn$. This says that $d\eta$ cannot be sympletic. This  was already known, see for instance \cite{GR2}. 
	\end{rem}
		
\begin{example}\label{closedonh}
	Let $\hh_3$ denote the Heisenberg Lie algebra of dimension three. 	Let $e^i$, i=1,2,3 denote the dual basis of the orthonormal basis $e_1, e_2, e_3$. With the convention $e^{ij}=e^i\wedge e^j$, clearly the 2-forms $e^{12}, e^{13}, e^{23}$ give a basis of the space of 2-forms on $\nn$.  Since $\la j(e_3)e_1, e_2\ra=\la e_3, [e_1,e_2]\ra=1$, then $e^{12}$ is exact. 
	
	It is not hard to see that any 2-form $\omega=\alpha e^{12} + \beta e^{13} + \gamma e^{23}$ is closed. 
\end{example}

Notice that any skew-symmetric map of type II, namely  $\lf_2$, gives rise to a 2-form $\omega$ satisfying the condition 
$$\omega(Z,\widetilde{Z})=0, \quad \mbox{ for all } Z, \widetilde{Z}\in \zz.$$
This means, that the center is ``isotropic'' for $\omega$. In this situation only one of  the two following conditions is true:
\begin{enumerate}[(i)]
		\item either $\omega(\zz, \nn)=0$ or
	\item there is $U\in \nn-\zz$ such that $\omega(Z, U)\neq 0$ for some $Z\in \zz$. 
\end{enumerate}
Indeed the non-trivial cases correspond to the second condition (ii). Due to the correspondence between 2-forms and skew-symmetric maps, one can obtain a description of 2-forms in terms of skew-symmetric maps of type I or II, as follows. 

As above, let $\la\,,\,\ra$ denote a metric on the 2-step nilpotent Lie algebra $\nn$ and take the orthogonal decomposition $\nn=\vv \oplus \zz$ as in Equation \eqref{decomp2}. Let $F:\nn \to \nn$ be the skew-symmetric map on $\nn$ such that $\omega(V+Z, \widetilde{V}+\widetilde{Z})=\la \lf(V+Z),\widetilde{V}+\widetilde{Z}\ra$. 

The condition of the center to be isotropic says that $\la \lf(Z),\widetilde{Z}\ra=0$ for all $Z, \widetilde{Z}\in \zz$, that is $\lf_{\zz}(Z)=0$, for any $Z\in \zz$,  which in terms of the families we introduce previously, gives:
$$\lf(Z+V)=\lf_1(V)+\lf_2(Z)+\lf_2(V)\quad \mbox{ for all } V+Z\in \nn.$$
But in this situation, the corresponding 2-form $\omega$ is closed if and only if
 $\lf_2$ satisfies Condition (C2). In fact, any skew-symmetric linear map $\lf_1$ with $\lf_1(\zz)\equiv 0$ induces a  closed 2-form as already noticed  in Remark \ref{ztrivial}. 

On the other hand, $\omega(Z,U)\neq 0$ if and only if $\la \lf(Z), U\ra \neq 0$ if and only if $\la \lf_2(Z),U\ra \neq 0$ for some $Z\in \zz$ and some $U\in \nn-\zz$.  And this must occur for any metric. 

\begin{prop}\label{prop2} Let $\nn$ denote a 2-step nilpotent Lie algebra. \begin{enumerate}[(i)]
			\item There is a non-trivial closed 2-form $\omega$ for which either  $\omega(\zz, \zz)\neq 0$ or $\omega(\zz, \nn)=0$ if and only if  for any metric $\la\,,\, \ra$ on the Lie algebra $\nn$ there is a non-trivial skew-symmetric map  of type {\rm I} satisfying Condition {\rm(C1)}.
		\item There is a non-trivial closed 2-form $\omega$ for which $\omega(\zz, \zz)=0$ but $\omega(\zz, \nn)\neq 0$ if and only if for any metric $\la\,,\, \ra$ on the Lie algebra $\nn$ there is a non-trivial skew-symmetric map  of type {\rm II}  satisfying Condition {\rm (C2)}.
	
	\end{enumerate}
	\end{prop}

 Moreover, in terms of Definition \ref{def1}, the proposition above enables to distinguish 2-forms into two families. We shall say that a  2-form $\omega$ on the Lie algebra $\nn$ is
 \begin{enumerate}[(i)]
 	\item  {\em of type} { \rm I}: if it satisfies  either  $\omega(\zz, \zz)\neq 0$ or $\omega(\zz, \nn)=0$
 	\item  {\em of type} {\rm{II}}: if it satisfies $\omega(\zz, \zz)=0$ but $\omega(\zz, \nn)\neq 0$.
 \end{enumerate}

 \begin{example} \label{hcomplex}
 	Let $\hh(\CC)$ denote the Heisenberg Lie algebra over $\CC$. Consider the underlying real Lie algebra of dimension six, that we denoted in the same way. It has a center of dimension two spanned by $Z_1,Z_2$ and complementary subspace of dimension four spanned by the vectors $X_1,Y_1,X_2, Y_2$. They satisfy the non-trivial Lie bracket relations
 	$$[X_1,Y_1]=-[X_2,Y_2]=Z_1\qquad [X_1,Y_2]=[X_2,Y_1]=Z_2.$$
 	The left multiplication by $i$ induces a real linear map $J:\hh(\CC) \to \hh(\CC)$ satisfying $J^2=-Id$ and $J\circ \ad(U)=\ad(U)\circ J$ for all $U\in \hh(\CC)$. Explicitly, in the basis, one has
 	$$J(Z_1)=Z_2 \quad J(X_1)=X_2 \quad J(Y_1)=Y_2.$$
 	Take the metric on $\hh(\CC)$ making  the set $Z_i,X_i,Y_i$ for  i=1,2, an orthonormal basis. Thus, $\hh(\CC)$ is a Lie algebra of type $H$.  Clearly the complex structure $J$ is skew-symmetric with respect to this metric $\la\,,\,\ra$.  Since this Lie algebra is non-singular, the kernel of $j$  is trivial, $\ker j=\{0\}$.
 	
 	Assume that $\lf$ is a skew-symmetric map on the Lie algebra giving rise to a closed 2-form. Write $\lf=\lf_1+\lf_2$ as above.  Since $\hh(\CC)$ is non-singular, Condition (C1) imposes  the linear map  $\lf_1$  is trivial on $\zz$. Thus, $\lf_1$ corresponds to a skew-symmetric map $\vv \to \vv$. 
 	
 	On the other hand, for   $V, W\in \vv$, the  condition (C2) for $\lf_2$ gives
 $$\begin{array}{rcl}
 	0 & = &\la \lf_2(V), [JV,W]\ra + \la \lf_2(JV), [W,V]\ra + \la \lf_2W, [V,JV]\ra\\
 &	= &\la \lf_2(V), J[V,W]\ra + \la \lf_2(JV), [W,V]\ra + \la \lf_2W, J[V,V]\ra\\
 &	= &\la \lf_2(V), J[V,W]\ra - \la \lf_2(JV), [V,W]\ra\\
 &	= & \la -J\lf_2(V)-\lf_2(JV), [V,W]\ra.
 \end{array}$$
 Since $W$ is an  arbitrary element and $\ad(V)$ is onto the center $\zz$ we get $	\lf_2(JV)=-J\lf_2(V)$,  for every $ V\in \vv$. And since $\lf_2$ is skew-symmetric it holds on $\nn$: 
 \begin{equation}\label{anticomp}
 	\lf_2 \circ J=-J\circ \lf_2.  
 \end{equation}
 Conversely,  any skew-symmetric map $\lf_2:\hh(\CC)\to \hh(\CC)$ that verifies Equation  \eqref{anticomp} will give rise to a closed 2-form of type II.
 \end{example}

 The next result proved in \cite{OS} determines a condition on the dimension of the Lie algebra and its center for the non-existence of  magnetic fields of type II, that is, the corresponding  skew-symmetric map is of type II. 
 
 \begin{lem} \label{closedF1} \cite{OS} 	Let $\nn$ denote a non-singular 2-step nilpotent Lie algebra such that $\dim \nn > 3 \dim \zz$. Then any closed 2-form on $\nn$ satisfies 
 	$$\omega(Z,U)=0, \quad \mbox{ for all } Z\in \zz, U\in \nn.$$
 \end{lem}
\begin{proof}
  	Firstly, note that since the Lie algebra is non-singular, its commutator coincides with the center,  $C(\nn)=\zz$. Let $\omega$ denote a closed two-form on $\nn$, the non-singularity  property implies that  $\omega(\zz,\zz)=0$. In fact, let $Z,  \widetilde{Z}\in \zz$ with $Z=[U,V]$ for $U,V\in \nn$. The closedness condition says that $\omega(\widetilde{Z}, [U,V])=0$. 
 	
 	By contradiction, assume that there exists a non-trivial closed 2-form $\omega$ on the non-singular 2-step nilpotent Lie algebra $\nn$   such that there are  $Z\in \zz$ and  $U\in \nn-\zz$ satisfying  $\omega(Z,U)\neq 0$, that is $\omega$ is of type II. 	
 	
 	Let $\la\,,\,\ra$ be a metric on $\nn$ inducing a orthogonal decomposition   $\nn=\vv\oplus \zz$ as in Equation \eqref{decomp2}. Notice that by hypothesis, $\dim\vv> 2\dim \zz$. 
 	
 	Let $\lf$ denote the skew-symmetric map on the Lie algebra $\nn$ such that $\omega(X,Y)=\la \lf(X), Y\ra$. Then 
 	$$\omega(Z, U)\neq 0 \mbox{ if and only if } \la \lf(Z),U \ra \neq 0 \mbox{ if and only if }\la \lf_2(Z),U\ra\neq 0,$$
 	where $\lf_1$ and $\lf_2$ are  taken as in Proposition \ref{prop1}. Since $\la Z, \lf_2(U)\ra \neq 0$ says that the image of $\lf_2|_{\vv}$ is non-trivial,  we may assume (changing $Z$ if necessary) that $\lf_2(U)=Z$, so that $\la Z, \lf_2(U)\ra = \la Z, Z\ra\neq 0$. 
 	
 Denote the kernel of $\lf_2$ by  $\mathcal W=\ker(\lf_2|_{\vv})$. Thus one has:
 	$$\dim \mathcal W=\dim \vv - \dim Image(\lf_2|_{\vv})\geq \dim \vv - \dim \zz > \dim \vv/2$$
 	
 	Since $j(Z)$ is non-singular, it  holds $\dim j(Z)\mathcal W>\dim \vv /2$, which implies that the intersection is nontrivial,  $\mathcal W \cap j(Z) \mathcal W\neq \{0\}$. Let elements $W, \widetilde{W}\in \mathcal W$ such that $j(Z)W=\widetilde{W}\neq 0$. Now, the closedness condition for the 2-form $\omega$ is equivalent to Condition (C2) for $F_2$. And for $W, \widetilde{W}, U$ we get
 	$$0=\la \lf_2(W), [\widetilde{W}, U]\ra+\la \lf_2(\widetilde{W}),[U,W]\ra+\la \lf_2(U), [W,\widetilde{W}]\ra= \la Z, [W, \widetilde{W}]\ra = \la \widetilde{W}, \widetilde{W}\ra\neq 0,$$
 	which is a contradiction. Thus, there are no closed 2-forms of type II under the hypothesis. 
 \end{proof}

  \begin{example} {\bf A singular example.} Let $\nn$ denote the  Lie algebra $\nn=\vv \oplus \zz$, where $\vv$ is spanned by the vectors $V_1, V_2, V_3, V_4, V_5$ and $\zz$ is spanned by $Z_1, Z_2$, and they obey the non-trivial Lie bracket relations:
  $$[V_1, V_2]=Z_1, \qquad [V_3,V_4]=Z_2=[V_4,V_5].$$
  Take the metric on $\nn$ that makes of the previous basis an orthonormal basis. 
  
  Let $\lf: \nn \to \nn$ denote the skew-symmetric map given by
  $$\lf(V_3)=Z_2 =-F(V_5), \qquad \lf(V_4)=Z_1, \quad \lf(V_i)=0, i=1,2.$$
  Usual computations show that the 2-form given as $\omega(X,Y)=\la \lf(X),Y\ra$ is closed. 
  
  This is an example of a singular Lie algebra (that is, every $j(Z)$ is singular) admitting a closed 2-form of type II. In fact, every map $j(z_1Z_1 + z_2 Z_2): \vv \to \vv$ has a matrix of the form
  $$\left( \begin{matrix}
  	0 & -z_1 & 0 & 0 & 0 \\
  	z_1 & 0  &  0 & 0 & 0\\
  	0 & 0 & 0 & -z_2 & 0 \\
  	0 & 0 & z_2 & 0 &-z_2 \\
  	0 & 0 & 0  & z_2 & 0
  \end{matrix}
  \right)
  $$
  in the basis of $\vv$ given above. And $\dim \nn> 3 \dim \zz$.

  \end{example}

 In \cite{OS} the authors determine the  algebras of type H admitting Lorentz forces of type II. 
  
 \begin{thm}\label{thm2}  Let $\nn=\vv\oplus\zz$ be a Lie algebra of type H. Then $\nn$ admits  a Lorentz force of type {\rm II} if and only if $\nn$ is the $3$-dimensional Heisenberg algebra, the $6$-dimensional complex Heisenberg algebra or the $7$-dimensional quaternionic Heisenberg algebra.
 \end{thm}

The results of this section give obstructions for the existence of left-invariant symplectic structures on 2-step nilpotent Lie groups. This follows as an application. 

\begin{cor}
Let $N$ denote a 2-step nilpotent Lie group of dimension $2n$. Whenever the Lie group  $N$ is non-singular and $\dim \nn > 3 \dim \zz$,  there is no symplectic structure on $N$. 

In particular, the only Lie group of type H admitting a left-invariant symplectic structure corresponds to the complex Heisenberg with Lie algebra $\hh(\CC)$.   
\end{cor}

\begin{proof} By Lemma \ref{closedF1} any non-singular Lie algebra $\nn$ satisfying $\dim \nn > \dim \zz$ cannot admit a closed-symplectic structure of type II. Thus any closed 2-form corresponds to a 2-form of type I. Such  2-form corresponds to a skew-symmetric map on $\nn$ satisfying Condition (C1), that is $\lf_{\zz}(Z)\in \ker{j}$ for all $Z\in \zz$. Thus by the non-singularity property $\lf_{\zz}(Z)=0$ for all $Z\in\zz$. So,  any closed 2-form of type  I is degenerate. 

  Among Lie algebras of type H admitting closed 2-form of type II, only the complex Heisenberg Lie algebra $\hh(\CC)$ is even-dimensional. According to computations in Example \eqref{hcomplex}, any Lorentz force on $\hh(\CC)$ has a matrix presentation of the form in the basis $X_1, X_2, Y_1, Y_2, Z_1, Z_2$:
  $$\left(\begin{matrix}
  	0 & a & b & c & g & h\\
  	-a & 0 & d &  e & h & -g\\
  	-b & -d & 0 & f & i & j\\
  	-c & -e & -f & 0 & j & -i\\
  	-g & -h & - i & - j & 0 & 0\\
  	-h & g & -j & i & 0 & 0 
  \end{matrix}
  \right) 
  $$
  and it is not difficult to see that there are non-singular examples. Thus, there exist symplectic structures in this case. 
\end{proof}

\begin{rem} Dotti and Tirao proved in \cite{DT}:
	
	\smallskip 
	
	{\em 	Let $M = T^t \times \Gamma \backslash H$ be an even dimensional $H$-type nilmanifold,
	where $H$ is an $H$-type group whose Lie algebra is not isomorphic to $\hh_3$, $\hh(\CC)$ or $\hh(\mathbb H)$. Then there
	is no symplectic structure on $M$.}	

By Remark \ref{remtypeIandII}, if $t\geq1$ we get examples of Lie groups that do not admit symplectic structures but  admitting  a non-trivial closed left-invariant 2-form of type II. In \ref{sec.example} we show more interesting examples on  indecomposable Lie algebras.
%\begin{proof} See \cite{M} and \cite{K}.
%\end{proof}
\end{rem}

 \section{Closed 2-forms on low dimensional Lie algebras}
 
 In this section we study closed 2-forms of type II on 2-step nilpotent Lie algebras of dimension $n$ with $n\leq 6$. The main goal is to compare the existence of closed 2-forms of type II with the existence of symplectic structures. 
 
 %In particular we prove:

  Notice that the non-existence  of closed 2-forms of type II on the Heisenberg Lie algebra of dimension $5$ is a consequence of Lemma \ref{closedF1}, while the existence for non-singular Lie algebras such as the Heisenberg Lie algebra of dimension 3, and the complex Heisenberg Lie algebra in dimension six, is stated in Theorem \ref{thm2}. 
  
  Firstly, recall the list of 2-step nilpotent Lie algebras, where we write the non-trivial Lie brackets:
  
  Dimension 3	
  	
  \begin{enumerate}
  		
  \item $\hh_1(\RR)$ Heisenberg Lie algebra of dimension three: $[e_1, e_2]=e_3$.

  \end{enumerate}	
  
  Dimension 4

  \begin{enumerate}
  	
  	\item $\hh_1(\RR)\oplus\RR=span\{e_1, e_2, e_3, e_4\}$
  with $[e_1, e_2]=e_3$.	
  \end{enumerate}		
  	
  Dimension 5
  
  \begin{enumerate}
  
  	\item $\hh_1(\RR)\oplus\RR^2=span\{e_1, e_2, e_3, e_4, e_5\}$
  	  with $[e_1, e_2]=e_3$.
  	
  	\item $\hh_2(\RR)$ Heisenberg Lie algebra of dimension $5$: $span\{e_1, e_2, e_3, e_4, e_5\}$
  	  with $[e_1, e_2]=[e_3,e_4]=e_5$.  
  	
  	\item $\ggo_5=span\{e_1, e_2, e_3, e_4, e_5\}$ with  $[e_1,e_3]=e_4$,  $[e_2,e_3]=e_5$, called    Star. 
  \end{enumerate}		
  
  Dimension 6: the Lie algebra is spanned by vectors $e_1, e_2, e_3, e_4, e_5, e_6$

  \begin{enumerate}
    	\item $\hh_1(\RR)\oplus\RR^3$ with $[e_1, e_2]=e_3$. 
  	
  	\item $\hh_2(\RR)\oplus\RR$ with $[e_1, e_2]=[e_3,e_4]=e_5$.
  	
  	\item $\ggo_5\oplus\RR$  with  $[e_1,e_3]=e_4$,  $[e_2,e_3]=e_5$.
  	
  	\item $\hh_1(\RR)\oplus\hh_1(\RR)$ with $[e_1,e_2]=e_5$, $[e_3, e_4]=e_6$. 
  	
  	\item $\ff_6$: where $[e_1,e_2]=e_4$  $[e_1,e_3]=e_5$ $[e_2,e_3]=e_6$. 
  	 This is the free 2-step nilpotent Lie algebra in three generators. %It is almost non-singular. 
  	   	
  	\item $\kk_6$: where $[e_1,e_4]=e_5$  $[e_2,e_3]=-e_5$ $[e_3,e_4]=e_6$. %$(4,2)$. %Almost non-singular
  	
  	\item $\hh_1(\CC)$ with $[e_1, e_2]=e_5=-[e_3, e_4]$, $[e_1, e_4]=e_6=[e_2, e_3]$. It is  the complex Heisenberg Lie algebra. % $(4,2)$  No singular
   	
  \end{enumerate}

  	   \begin{prop} Every 2-step nilpotent Lie group of dimension $n$ with $n\leq 6$ admits a non-trivial closed 2-form of type II, with exception of the real Heisenberg Lie algebra of dimension $5$. 
  	     \end{prop}
  	
  	\begin{proof}
  	For the general approach to the proof, let $\nn$ denote a 2-step nilpotent Lie algebra equipped with any metric. Then the Lie algebra admits a decomposition $\nn=\vv \oplus \zz$ as in \ref{decomp2}. 
  	
  	Let $v_1, \hdots v_m$ be a orthonormal basis of $\vv$ and $z_1, \hdots z_n$ a orthonormal basis of $\zz$. As usual, write $C_{ij}^k$ for the constant structures of the Lie algebra, which in our cases is
  	$$[v_i, v_j]=\sum_{s=1}^n C_{ij}^s z_s.$$
  	
  	Now, let $\omega$ denote a closed 2-form of type II, which satisfy $\omega = \la F \cdot, \cdot \ra$. Thus $\lf(v_k)=\sum_{t=1}^n b_{tk} z_t$ for any $k=1, \hdots m$. Notice that $b_{sk}=-b_{ks}$ from the skew-symmetric property of $F$. Now, the closedness condition says 
  	$$\la [v_i, v_j], \lf(v_k)\ra + \la [v_j, v_k], \lf(v_i)\ra + \la [v_k, v_i], \lf(v_j)\ra, \quad \mbox{ for all }i,j,k, $$ 
  	which explicitly gives
  	 $$ \sum_{s=1}^n (C_{ij}^s b_{sk} +  C_{jk}^s b_{si} + C_{ki}^s b_{sj})=0.$$
  	 Indeed trivial solutions $b_{si}=0$ for all $i,\,s$ are always possible. 
  	 
  	 The next table shows the solutions of the equation above in every case. 
  	 
  	 We denote by $e^i$, the dual forms of the elements of the $e_j$ of the corresponding Lie algebras. And as usual $e^{ij}:= e^i \wedge e^j$.
  	 
  	 \begin{table}[htb]
  	 \begin{tabular}{|c|rc|}
  	\hline
  	Lie algebra & closed 2-forms & \\ \hline
  	$\hh_1$ & $ a_{13} e^{13}+ a_{23} e^{23},$ & $a_{i3}\in \RR$, for $i=1,2$. \\ \hline
  		$\hh_1\oplus \RR$ & $ a_{13} e^{13}+ a_{23} e^{23}+ a_{14} e^{14}+ a_{24} e^{24},$ & $a_{ij}\in \RR$. \\ \hline
  $\hh_1\oplus \RR^2$ & $ \sum_{j=3}^5a_{1j} e^{1j}+ \sum_{j=3}^5a_{2j} e^{2j},$ & $a_{ij}\in \RR$. \\ \hline
  $\hh_2$ & no & \\ \hline
  $\ggo_5$ &  $\sum_{i=1}^3\sum_{j=4}^5a_{ij} e^{ij},$ & $a_{ij}\in \RR$. \\ \hline
   $\hh_1\oplus \RR^3$ & $ \sum_{j=3}^6a_{1j} e^{1j}+ \sum_{j=3}^6a_{2j} e^{2j},$ & $a_{ij}\in \RR$. \\ \hline	
   $\hh_2\oplus \RR$ & 	$\sum_{i=1}^4a_{i6j} e^{i6},$ & $a_{i6}\in\RR$. \\ \hline
   $\ggo_5 \oplus \RR$ & $\sum_{i=1}^3\sum_{j=4}^6a_{ij} e^{ij},$ & $a_{ij}\in \RR$. \\ \hline
    $\hh_1\oplus \hh_1$ & $ \sum_{i=1}^2a_{i5} e^{i5}+ \sum_{j=3}^4a_{i6} e^{i6},$ & $a_{ij}\in \RR$. \\ \hline	
    $\mathfrak f_6$ & $\sum_{i=1}^3 \sum_{j=4}^6a_{ij} e^{ij}, a_{25}=a_{16}+a_{34},$ & $a_{ij}\in \RR$. \\ \hline
    $\mathfrak k_6$ & $a (e^{16}+e^{35})+ b (e^{26}+e^{45})+\sum_{i=3}^4a_{i6} e^{i6}, $ & $a_{ij}\in \RR$. \\ \hline	
    $\hh(\CC)$ & $\sum_{\substack{i=1,3}}a_{i5} (e^{i5}-e^{i+1 6})+ \sum_{\substack{i=1,3}}a_{i6} (e^{i6}+e^{i+1 5}), $ & $a_{ij}\in \RR$. \\ \hline
  		  	 \end{tabular}
  		  	 \caption{Closed 2-forms of type II}
  \end{table}	 
  	   	\end{proof}
  	   	
%In the following paragraphs we exhibit  the possible closed 2-forms on every Lie algebra. They are presented as skew-symmetric matrices in the basis listed above. 
Let us now determine the 2-closed forms on
each of the nilpotent Lie algebras of dimension $n\leq 6$ listed above. They are obtained from canonical computations by asking the 2-form to be closed (see \eqref{clos}). 
 
 Recall that $\nn =\vv \oplus \zz$ with $\vv=\zz^{\perp}$. And any skew-symmetric map $\lf$ of type I preserves the decomposition, while if $\lf$ is of type II, then $\lf(\zz)\subseteq \vv$ and $\lf(\vv)\subseteq \zz$. 
 
 \smallskip
 
   Dimension 3: In $\hh_1(\RR)$ every 2-form is closed and it is represented by any skew-symmetric map $F:\hh_1(\RR)\to \hh_1(\RR)$. Thus in the canonical basis, it has a matricial presentation
   $$\left( \begin{matrix}
   A & B\\
   -B^t & 0 
   \end{matrix} \right) \quad A\in \mathfrak{so}(2,\RR), \quad  B^t=\left( \begin{matrix} b & c  \end{matrix} \right).
   $$

   Dimension 4: In $\hh_1(\RR)\oplus \RR$ every closed 2-form is represented by a matrix of the form:
       $$\left( \begin{matrix}
      A & B \\
     -B^t & 0 
      \end{matrix}  \right)\quad A\in\mathfrak{so}(2, \RR), \quad  B^t=\left( \begin{matrix} b & c \\ d & e  \end{matrix} \right). 
      $$
      Notice that there are symplectic structure, whenever the matrix above is non-singular. 
      
      Dimension 5: 
      
       $\bullet$ In $\hh_1(\RR)\oplus \RR^2$   every closed 2-form is represented by a matrix of the form:
       $$\left( \begin{matrix}
           A & B\\
           -B^t & C
            \end{matrix}  \right)\quad A\in \mathfrak{so}(2
            ,\RR), \quad B^t \in M_{2\times 3}(\RR) \quad C = \left( \begin{matrix} 
                        0 & 0 & 0 \\
                        0 & 0 & h \\
                        0 & -h & 0
                        \end{matrix}\right). 
            $$
            
                   $\bullet$ In $\hh_2(\RR)$   every closed 2-form is represented by a matrix of the form:
             $$\left( \begin{matrix}
                 A & 0 \\
                 0 & 0
                  \end{matrix}  \right)\quad A\in \mathfrak{so}(4,\RR). 
                  $$
                  
              $\bullet$    In $\ggo_5$ every closed 2-form is represented by a matrix of the form:
              $$\left( \begin{matrix}
              A & B \\
              -B^t & 0
              \end{matrix} \right) \quad A\in \mathfrak{so}(3,\RR),\quad  B^t\in M_{2\times 3}(\RR).$$

            Dimension $6$
            
             $\bullet$ In $\hh_1(\RR)\oplus\RR^3$ every closed 2-form is represented by a matrix of the form:
             $$\left( \begin{matrix}
             A & B \\
             -B^t & C
             \end{matrix} \right) \quad A\in \mathfrak{so}(2,\RR),\quad  B\in M_{2\times 4}(\RR), \quad   C =\left( \begin{matrix}
             0 & 0 & 0 & 0 \\
             0 & 0 & -e & -f \\
             0 & e & 0 & -g\\
             0 & f & g & 0
             \end{matrix} \right).$$
            
            The 2-form $\omega=e^1\wedge e^3 + e^2\wedge e^4 + e^5 \wedge e^6$ gives a symplectic structure. 
            
            $\bullet$ In $\hh_2(\RR)\oplus\RR$ every closed 2-form is represented by a matrix of the form:
            $$\left( \begin{matrix}
            A & B \\
            -B^t & 0
            \end{matrix} \right) \quad A\in \mathfrak{so}(4,\RR),\quad  B^t=\left( \begin{matrix}
            0 & 0 & 0 & 0 \\
            h & g & j & i
            \end{matrix} \right).$$
            There is no symplectic structure. 
            
            $\bullet$  In $\ggo_5\oplus\RR$ every closed 2-form is represented by a matrix of the form:
            $$\left( \begin{matrix}
            A & B \\
            -B^t & 0
            \end{matrix} \right) \quad A\in \mathfrak{so}(3,\RR),\quad  B\in M_{3\times 3}(\RR).$$
            
            The 2-form $\omega=e^1\wedge e^4 +e^2 \wedge e^5 + e^3 \wedge e^6$ gives a symplectic structure.

 $\bullet$ In $\hh_1(\RR) \oplus \hh_1(\RR)$ every closed 2-form is represented by a $6 \times 6$ matrix of the form:
   $$\left( \begin{matrix}
    A & B \\
    -B^t & 0
   \end{matrix} \right) \quad  A\in \mathfrak{so}(4,\RR),\quad  B^t = \left( \begin{matrix} a & b & 0 & 0 \\
   0 & 0 & c & d
   \end{matrix} \right).$$
  
  The 2-form $\omega=e^1\wedge e^5 +e^2 \wedge e^4 + e^3 \wedge e^6$ gives a symplectic structure.  
  
   $\bullet$  In $\ff_6$ every closed 2-form is represented by a matrix of the form:
  $$\left( \begin{matrix}
  	A & B \\
  	-B^t & 0
  \end{matrix} \right) \quad A\in \mathfrak{so}(3,\RR),\quad  B =\left( \begin{matrix}
  	d & e & f \\
  	g & f+i & h\\
  	i & j & k
  \end{matrix} \right).$$
  
  The 2-form $\omega=e^1\wedge e^6 +2 e^2 \wedge e^5 + e^3 \wedge e^4$ gives a symplectic structure.  
  
   $\bullet$ In $\kk_6$ every closed 2-form is represented by a $6 \times 6$ skew-symmetric  matrix of the form:
   $$\left( \begin{matrix}
   A & B \\
   -B^t & 0
   \end{matrix} \right) \quad  A\in \mathfrak{so}(4,\RR),\quad  B^t =\left( \begin{matrix}
   0 & 0 & g & h \\
   g & h & i & j
   \end{matrix} \right).$$
   
   The 2-form $\omega=e^1\wedge e^6 +e^2 \wedge e^4 + e^3 \wedge e^5$ gives a symplectic structure.  
   
   $\bullet$  In $\hh_1(\CC)$ every closed 2-form is represented by a matrix of the form:
   $$\left( \begin{matrix}
   A & B \\
   -B^t & 0
   \end{matrix} \right) \quad A\in \mathfrak{so}(4,\RR),\quad  B^t=\left( \begin{matrix}
   g & h & i & j \\
   h & - g & j & -i
   \end{matrix} \right).$$
   
     The 2-form $\omega=e^1\wedge e^6 +e^2 \wedge e^5 + e^3 \wedge e^4$ gives a symplectic structure.  
     
   The results of the previous paragraphs enables the proof of the following result. 
   
    \begin{thm} Every 2-step nilpotent Lie group of dimension $2n$ with $n\leq 3$ admits a left-invariant symplectic structure if and only if it admits a  closed left-invariant 2-form of type II, with exception of the trivial extension of the Heisenberg Lie group of dimension five. 
     \end{thm}

   \subsection{The action of the group of orthogonal automorphisms}
   
   As proved in \cite{OS2} there is an action of the group of orthogonal  automorphisms that facilitates to find the solutions: in fact, assume $\gamma$ is a solution of Equation \eqref{magneteq}. Take a orthogonal automorphism $\varphi$, then 
   $$\nabla_{d\varphi \gamma'} d\varphi \gamma'= d\varphi \nabla_{\gamma'} \gamma' = d\varphi F \gamma' = d\varphi F d\varphi^{-1} d\varphi \gamma',$$
   see more details in \cite{OS2}. %Thus, the action of the group of orthogonal automorphisms facilitates to find the magnetic trajectories. 
   Thus, the magnetic trajectories for $F$ are related via $d\varphi$  with the magnetic trajectories for $d\varphi F d\varphi^{-1}$. 
  
  Orthogonal automorphisms of the 2-step nilpotent Lie group $N$ are in correspondence with orthogonal automorphisms of the correspondent Lie algebra $\nn$. That is, $\psi: \nn\to \nn$, such that $\la \psi x, \psi y\ra=\la x, y\ra$ for any $x,y\in \nn$.
  
  By using the orthogonal decomposition $\nn=\vv\oplus \zz$ we have that 
  \begin{itemize}
  \item $\psi(\zz)\subseteq \zz$ and $\psi(\vv)\subseteq \vv$,
  \item $\psi=(A,B)$ with $A\in \mathrm{O}(\vv)$ and $B\in \mathrm{O}(\zz)$ satisfy $Aj(Z)A^{-1}=j(BZ)$, for all $Z\in \zz$.
  \end{itemize}
  In particular, the last equality above says that exact magnetic fields are preserved by this action. 
  
  It is not hard to see  that this action preserves the type of the magnetic field. That is, if $\lf$ is of type I (resp. II), then $\psi  \circ \lf \circ \psi^{-1}$ is of type I (resp. II), for any $\psi:\nn \to \nn$ orthogonal automorphism.
  
  \smallskip
   
  \begin{rem} In \cite{OS2} the authors considered a more general action to compute solutions. In fact given a magnetic trajectory $\gamma: I \to M$ solution of the magnetic equation \eqref{magneteq} for a Lorentz force $F$, then the curve $\phi\circ \gamma(\nu s)$ is solution of the magnetic equation for the Lorentz force $\nu  d\phi \lf d\phi^{-1}$, for $\nu \in \RR-\{0\}$ and $\phi$ a isometry of $M$.  
  \end{rem}
   
  \begin{example} In the Heisenberg Lie algebra of dimension three the group of orthogonal automorphisms $Auto(\hh_1(\RR))$ consists of maps $\psi:\nn \to \nn$, where $\psi(v+z)=Av+ det(A)z$, for $v\in \vv, z\in \zz$, $A\in \mathrm{O}(2)$. 
   
   Under this action any closed 2-form
  \begin{itemize}
  \item of type I is equivalent to  $j(e_3)$;
  \item  of type II it is equivalent to 
        $$\left( \begin{matrix}
        0 & B\\
        -B^t & 0 
        \end{matrix} \right)   \quad  B^t=\left( \begin{matrix} r & 0 \end{matrix} \right) \quad r\in \RR^+.$$
  \end{itemize} 
   The action of $Auto(\hh_1(\RR))$ preserves the norm $\Vert v \Vert$ for any $v\in\vv$. Thus $r=\Vert v \Vert$.    
    \end{example}
     
  The natural action is that  of the automorphism group on  the second cohomology group $H^2(\nn, \RR)$ which enables to classify the extensions of $\nn$. See for instance \cite{MJ}.

  \subsection{Lie algebras admitting closed 2-forms of type II but no symplectic structure.}\label{sec.example} Now we present examples of indecomposable Lie algebras admitting a closed 2-form of type II but no symplectic structure. 
  
   There is a family of 2-step nilpotent Lie algebras which are constructed from graphs.   
  Let $G$ be a directed graph with at least one edge. Denote the   vertices of $G$ by $S = \{X_1, . . . ,X_m\}$ and its edges by $E = \{Z_1, . . . ,Z_q\}$.
    The Lie algebra $L(G)$ is the vector space direct sum $\vv \oplus \zz$, where we let $E$
    be a basis over $\RR$ for $\zz$ and $S$ be a basis over $\RR$ for $\vv$. Define the bracket relations
    among elements of $S$ according to adjacency rules:
    \begin{itemize}
    \item if $Z_k$ is a directed edge from vertex $X_i$ to vertex $X_l$ then define the skew-symmetric
      bracket $[X_i,X_l] = Z_k$.
   \item If there is no edge between two vertices, then define the bracket of those
      two elements in $S$ to be zero.
    \end{itemize}

    Extend the bracket relation to all of $\vv$ by using bilinearity of the bracket. The existence problem of symplectic structures in this family of 2-step nilpotent Lie algebras was solved by   Pouseele and Tirao  in \cite{PT}:  
  
  \smallskip
  
 \begin{thm} \cite{PT}.Let $G$ be a graph. Then the Lie algebra $L(G)$ associated with $G$ is symplectic if and only if $|V|+ |E|$ is even, and, in
  each connected component of $G$, the number of edges does not exceed the number of vertices.
  \end{thm}
  
  \smallskip
  
  Consider the complete graph $K_4$: it has four vertices $V_1, V_2, V_3, V_4$ and six edges $Z_1, Z_2, Z_3$, $Z_4,Z_5,Z_6$, giving rise to a Lie algebra of dimension ten with non-trivial Lie brackets
  $$[V_1, V_2]=Z_{1},\, [V_1,V_3]=Z_{2},\, ,\, [V_1,V_4]=Z_{3},\, [V_2,V_3]=Z_{4},\, [V_2,V_4]=Z_{5},\, [V_3,V_4]=Z_{6}.$$
  Since $|E|=6>4=|V|$ the corresponding Lie algebra does not admit a symplectic structure. 
  
  However it admits a 2-form of type II: one has to solve four linear equations for the coordinates $a_{ij}$ such that $F(V_i)=\sum_{j=1}^6 a_{ji} Z_{j}$:
  
 \begin{itemize}
 \item $\la [V_1,V_2], F(V_3)\ra + \la [V_2,V_3], F(V_1)\ra + \la [V_3,V_1], F(V_2)\ra =0$ is equivalent to
 
 $a_{13} + a_{41} - a_{22} =0$, 
 
 \item  $\la [V_1,V_2], F(V_4)\ra + \la [V_2,V_4], F(V_1)\ra + \la [V_4,V_1], F(V_2)\ra =0$ is equivalent to
  
  $a_{14} + a_{51} - a_{32} =0$,
  
   \item  $\la [V_1,V_3], F(V_4)\ra + \la [V_3,V_4], F(V_1)\ra + \la [V_4,V_1], F(V_3)\ra =0$ is equivalent to
      
      $a_{24} + a_{61} - a_{33} =0$.
      
  \item  $\la [V_2,V_3], F(V_4)\ra + \la [V_3,V_4], F(V_2)\ra + \la [V_4,V_2], F(V_3)\ra =0$ is equivalent to
   
   $a_{44} + a_{62} - a_{53} =0$.

 \end{itemize}
  This is a linear homogeneous system with 24 variables and 4 equations. That is, there always exists non-trivial solutions. Notice that the indices at the beginning of each equations show that the corresponding column has a 1 only in this position. Thus, the system has column rank equals four.  Moreover the result of Pouseele and Tirao says:
  
  \smallskip
  
  {\it Among complete graphs $K_n$, the only one giving rise to a  symplectic Lie algebra is $K_3$}.

  \smallskip
  
  The ideas of the previous  example for the complete graph $K_4$ can be extended to other complete graphs $K_n$ to get closed 2-forms of type II. One has to solve $\frac{n(n-1)(n-2)}{3!}$ equations for $n \frac{n(n-1)}2$ coefficients $a_{ij}^k$, for $n\geq 3$, where $i<j$. In fact the corresponding matrix has column rank equals $\frac{n(n-1)(n-2)}{3!}$, which is verified at the column of $a_{ij}^k$ for all possible $i,j,k$ with $i<j<k$, and $1\leq i$.
  
  The associated Lie algebra for $K_n$ is the free 2-step nilpotent Lie algebra in $n$ generators.

  \begin{rem} The Lie algebra induced from the graph $K_3$ is the Lie algebra $\mathfrak f_6$. It admits closed 2-forms of type II. Moreover there exists symplectic forms induced from 2-forms of type II. In fact, take for instance the 2-form induced from a matrix as in page 15 with $d=j=h=1$ and the other coefficients equal zero, is symplectic. 
  	\end{rem}


\begin{thebibliography}{GGGG}




\bibitem{BG}  {\sc C. Benson, C. S. Gordon}, {\it K\"ahler and symplectic structures on nilmanifolds}, Topology {\bf 27} (4), 513--518 (1988).

\bibitem{Bu} {\sc D. Burde}, {\it  Characteristically nilpotent Lie algebras and symplectic structures}, Forum Mathematicum
{\bf 18} (5), 769--
787 (2006).

\bibitem{Ca} {\sc A. C\v{a}p, J. Slovak}, {\it Parabolic geometries I},  Background and general theory, Math. Surveys 	and Monographs vol {\bf 154}, AMS (2009).

\bibitem{DM} {\sc J. M. Dardi\'e, A. Medina}, {\it  Double extension symplectique d’un groupe de Lie
	symplectique}, Adv. Math. {\bf 117}, 208--227  (1996).

\bibitem{dB1} {\sc V. del Barco}, {\it 
Symplectic structures on nilmanifolds: an obstruction for their existence}, 
J. Lie Theory {\bf 24} (3), 889--908 (2014). 

\bibitem{dB2} {\sc V. del Barco}, {\it 
Symplectic structures on free nilpotent Lie algebras}. 
Proc. Japan Acad., Ser. A {\bf 95} (8), 88--90 (2019). 


\bibitem{DT} {\sc I. Dotti, P. Tirao}, {\it Symplectic structures 	on Heisenberg-type nilmanifolds}, Manuscripta Math. {\bf 102}, 383--401 (2000).

\bibitem{Ch} {\sc B. Y. Chu}, {\it Symplectic homogeneous spaces}, Trans. Am. Math. Soc. {\bf 197}, 145--159 (1974).

\bibitem{Eb} {\sc P. Eberlein}, {\it Geometry of 2-step nilpotent Lie groups with a left-invariant metric}, Ann. Sci. E. N. S., 4 serie {\bf 27} (5),  611--660 (1994). 


\bibitem{FG} {\sc M. Fern\'andez, A.  Gray, A.}, {\it  Compact symplectic solvmanifolds not admitting
complex structures}, Geom. Dedicata {\bf 34}, 295--299  (1990).

\bibitem{GMK} {\sc J. G\'omez,  A. Jim\'enez-Merch\'an, and Y. Khakimdjanov}, {\it Symplectic structures on
	filiform Lie algebras}, J.  Pure and Applied Algebra, {\bf 156} (1), 15–-31  (2001).

\bibitem{GB} {\sc M. Goze,  A. Bouyakoub},{\it  Sur les alg\`ebres de Lie munies d’une forme symplectique}, Rend. Sem. Fac. Sci. Univ. Cagliari {\bf 57}(1), 85--97 (1987).



\bibitem{GR} {\sc 	M. Goze, E. Remm}, {\it Symplectic structures on 2-step nilpotent Lie algebras}, arXiv:1510.08212 (2015).


\bibitem{GR2} {\sc 	M. Goze, E. Remm}, {\it
Contact and Frobeniusian forms on Lie groups}, Differ. Geom. Appl. {\bf 35}, 74--94 (2014). 

\bibitem{K}{\sc A. Kaplan}, {\it Fundamental solutions for a class of hypoelliptic PDE generated by composition of quadratic forms}, Trans. Am. Math. Soc. {\bf 258},  147--153  (1980).


\bibitem{LO}{\sc J. Lauret, D. Oscari}, {\it On nonsingular two-step nilpotent Lie algebras}, Math. Res. Lett. {\bf 21} (3)
553--583 (2014).

\bibitem{LT}{\sc F. Levstein, A. Tiraboschi}, {\it Regular metabelian Lie algebras}, Geometry and Representation
Theory of Real and p-adic groups. Progress in Mathematics, vol. {\bf 158}, 197--207 (1998).

\bibitem{MJ} {\sc D. V. Millionshchikov, R. Jimenez}, {\it 
Geometry of central extensions of nilpotent Lie algebras}, 
Proc. Steklov Inst. Math. {\bf 305}, 209-231 (2019). Translation from Tr. Mat. Inst. Steklova 305, 225-249 (2019). 

\bibitem{M}{\sc R. Montgomery}, {\it A tour of subriemannian geometries, their geodesics and applications}, AMS Mathematical Surveys and Monographs 91 (2002).

\bibitem{N} {\sc K. Nomizu}, {\it On the cohomology of compact homogeneous spaces of nilpotent Lie
groups},  Ann. of Math. {\bf 59}, 531--538 (1954).

\bibitem{Ov} {\sc G. Ovando}, {\it Four Dimensional
Symplectic Lie Algebras}. Beitr.  Algebra  Geom.
{\bf  47}(2), 419--434  (2006).

\bibitem{OS} {\sc G. Ovando, M. Subils}, {\it Magnetic fields on non-singular 2-step nilpotent Lie groups}, ArXiv:2210.12180.

\bibitem{OS2} {\sc G. Ovando, M. Subils}, {\it Magnetic Trajectories on 2-Step Nilmanifolds},  J. Geom
Anal {\bf  33}, 186 (2023).

\bibitem{PT} {\sc H. Pouseele and P. Tirao}, {\it Compact symplectic nilmanifolds associated with graphs},  J. Pure Appl. Algebra {\bf 213} (9), 1788--1794 (2009).

%\bibitem{R}{\sc C. Riehm}, {\it The automorphism group of a composition of quadratic forms}, Transactions of the American Mathematical Society {\bf 269}, 403--414  (1982).

%\bibitem{Sc} {\sc M. Schneider} {\it Closed magnetic geodesics on $S^2$}, J. Differ. Geom. {\bf 87} (2), 343--388 (2011). 

%\bibitem{Su} {\sc T. Sunada}, {\it  Magnetic flows on a Riemann surface}, In: Proceedings of KAIST Mathematics Workshop,  93--108 (1993).

%\bibitem{Ta} {\sc I. Taimanov}, {\it  On an integrable magnetic geodesic flow on the two-torus}, Regul. Chaotic Dyn. {\bf 20} (6), 667--678 (2015). 



\end{thebibliography}
\end{document}